\documentclass[12pt]{amsart}

\usepackage[T1]{fontenc} 
\usepackage{amsmath}
\usepackage{amsthm} 
\usepackage{amscd} 
\usepackage{amssymb}
\usepackage{enumerate} 
\usepackage{latexsym} 
\usepackage{xypic}
\usepackage{dsfont} 
\usepackage{array} 
\usepackage{fullpage}
\usepackage{hyperref}

\usepackage{graphicx}
\usepackage{epsfig} 

\def\epsilon{\varepsilon}
 
\newcommand{\ssm}{\smallsetminus}

\renewcommand{\phi}{\varphi}

\DeclareMathOperator{\ind}{\mathrm{ind}} 
\DeclareMathOperator{\igeo}{\mathrm{ind_{geo}}}
\DeclareMathOperator{\iq}{\mathrm{ind}_{\CQ}}
\DeclareMathOperator{\indaut}{\mathrm{ind}}

\newcommand{\Out}{\mathrm{Out}} 
\newcommand{\Aut}{\mathrm{Aut}}
\newcommand{\Stab}{\mathrm{Stab}}
\newcommand{\inv}{^{-1}}

\newcommand{\rank}{\mathrm{rank}}

\newcommand{\Fix}{\mathrm{Fix}} 
\newcommand{\Per}{\mathrm{Per}} 
\newcommand{\Inn}{\mathrm{Inn}}

\newcommand{\Att}{\mathrm{Att}}

\newcommand{\FN}{F_N} 

\newcommand{\CVN}{\mathrm{CV}_N}
\newcommand{\barCVN}{\bar{\mathrm{CV}}_N} 
\newcommand{\CQ}{{\mathcal Q}}

\newcommand{\iwip}{iwip}
\newcommand{\FR}{FR}

\newcommand{\R}{\mathbb R} 
\newcommand{\Z}{\mathbb Z}
 
\newcommand{\N}{\mathbb N}

\def\bar{\overline} 
\def\tilde{\widetilde} 
\def\hat{\widehat}

\newtheorem{thm}{Theorem}[section]
 
\newtheorem{lem}[thm]{Lemma}
\newtheorem{prop}[thm]{Proposition}

\newtheorem*{thm*}{Theorem}
\newtheorem*{prop*}{Proposition}

\newtheorem*{thmbold}{Theorem~\ref{thm:boldconjecture}}
\newtheorem*{propsind}{Propositions~\ref{prop:geomauto} and \ref{prop:indexrank}}

\newtheorem*{thmiwipindec}{Theorem~\ref{thm:iwipindec}}

\newtheorem{defn}[thm]{Definition}
\newtheorem*{defn*}{Definition} 
\newtheorem{rem}[thm]{Remark}
\newtheorem*{rem*}{Remark}

\numberwithin{equation}{section} 

\newcommand{\arxivlink}[1]{\href{http://arxiv.org/abs/#1}{arXiv:#1}}

\begin{document}
\title{Botany of irreducible automorphisms of free groups}

\author{Thierry Coulbois, Arnaud Hilion}

\date{\today }

\begin{abstract}
  We give a classification
  of \iwip\ outer automorphisms of the free group, by discussing the
  properties of their attracting and repelling trees.
\end{abstract}

\maketitle

\section{Introduction}

An outer automorphism $\Phi$ of the free group $\FN$ is \textbf{fully
  irreducible} (abbreviated as \textbf{\iwip}) if no positive power
$\Phi^n$ fixes a proper free factor of $\FN$. Being an iwip is one 
(in fact the most important) of the analogs for free groups of being
pseudo-Anosov for mapping classes of hyperbolic surfaces. Another
analog of pseudo-Anosov is the notion of an atoroidal automorphism: an
element $\Phi\in\Out(\FN)$ is \textbf{atoroidal} or
\textbf{hyperbolic} if no positive power $\Phi^n$ fixes a nontrivial
conjugacy class. Bestvina and Feighn~\cite{bf-combi} and
Brinkmann~\cite{brink} proved that $\Phi$ is atoroidal if and only if
the mapping torus $\FN\rtimes_\Phi\Z$ is Gromov-hyperbolic.

Pseudo-Anosov mapping classes are known to be ``generic'' elements of
the mapping class group (in various senses). Rivin~\cite{rivin} and
Sisto~\cite{sisto} recently proved that, in the sense of random walks,
generic elements of $\Out(\FN)$ are atoroidal iwip automorphisms.

Bestvina and Handel~\cite{bh-traintrack} proved that iwip
automorphisms have the key property of being represented by
(absolute) train-track maps.

A pseudo-Anosov element $f$ fixes two projective classes of measured
foliations $[(\mathcal{F}^+,\mu^+)]$ and $[(\mathcal{F}^-,\mu^-)]$:
\[
(\mathcal{F}^+,\mu^+)\cdot f=(\mathcal{F}^+,\lambda\mu^+)\text{ and
}(\mathcal{F}^-,\mu^-)\cdot f=(\mathcal{F}^-,\lambda\inv\mu^-)
\]
where $\lambda>1$ is the expansion factor of $f$. Alternatively,
considering the dual $\R$-trees $T^+$  and
$T^-$, we get:
\[
T^+\cdot f=\lambda T^+\text{ and }T^-\cdot f=\lambda\inv T^-.
\]

We now discuss the analogous situation for iwip automorphisms.  The
group of outer automorphisms $\Out(\FN)$ acts on the \textbf{outer
  space} $\CVN$ and its boundary $\partial\CVN$. Recall that the
compactified outer space $\barCVN=\CVN\cup\partial\CVN$ is made up of
(projective classes of) $\R$-trees with an action of $\FN$ by
isometries which is minimal and very small. See \cite{vogt-survey} for
a survey on outer space.  An \iwip\ outer automorphism $\Phi$ has
North-South dynamics on $\barCVN$: it has a unique attracting fixed
tree $[T_\Phi]$ and a unique repelling fixed tree $[T_{\Phi\inv}]$ in
the boundary of outer space (see \cite{ll-north-south}):
\[
T_\Phi\cdot\Phi=\lambda_\Phi T_\Phi\text{ and
}T_{\Phi\inv}\cdot\Phi=\frac 1{\lambda_{\Phi\inv}}T_{\Phi\inv},
\]
where $\lambda_\Phi>1$ is the \textbf{expansion factor} of $\Phi$
(i.e. the exponential growth rate of non-periodic conjugacy classes).

Contrary to the pseudo-Anosov setting, the expansion factor
$\lambda_\Phi$ of $\Phi$ is typically different from the expansion
factor $\lambda_{\Phi\inv}$ of $\Phi\inv$. More generally, qualitative
properties of the fixed trees $T_\Phi$ and $T_{\Phi\inv}$ can be
fairly different. This is the purpose of this paper to discuss and
compare the properties of $\Phi$, $T_\Phi$ and $T_{\Phi\inv}$.

First, the free group, $\FN$, may be realized as the fundamental group
of a surface $S$ with boundary.  It is part of folklore that, if
$\Phi$ comes from a pseudo-Anosov mapping class on $S$, then its limit
trees $T_\Phi$ and $T_{\Phi\inv}$ live in the Thurston boundary of
Teichmüller space: they are dual to a measured foliation on the
surface. Such trees $T_\Phi$ and $T_{\Phi\inv}$ are called
\textbf{surface trees} and such an \iwip\ outer automorphism $\Phi$ is
called \textbf{geometric} (in this case $S$ has exactly one boundary
component).

The notion of surface trees has been generalized (see for instance
\cite{best-survey}). An $\R$-tree which is transverse to measured
foliations on a finite CW-complex is called \textbf{geometric}. It may
fail to be a surface tree if the complex fails to be a surface.

If $\Phi$ does not come from a pseudo-Anosov mapping class and if
$T_\Phi$ is geometric then $\Phi$ is called \textbf{parageometric}. For a
parageometric \iwip\ $\Phi$, Guirardel \cite{guir-core} and Handel and
Mosher \cite{hm-parageometric} proved that the repelling tree
$T_{\Phi\inv}$ is not geometric. So we have that, $\Phi$ comes from a
pseudo-Anosov mapping class on a surface with boundary if and only if
both trees $T_\Phi$ and $T_{\Phi\inv}$ are geometric. Moreover in this
case both trees are indeed surface trees.

In our paper \cite{ch-a} we introduced a second dichotomy for trees in
the boundary of Outer space with dense orbits. For a tree $T$, we
consider its \textbf{limit set} $\Omega\subseteq \bar T$ (where $\bar
T$ is the metric completion of $T$). The limit set $\Omega$ consists
of points of $\bar T$ with at least two pre-images by the map
$\CQ:\partial\FN\to\hat T=\bar T\cup\partial T$ introduced by Levitt
and Lustig~\cite{ll-north-south}, see Section~\ref{sec:botanictrees}.
We are interested in the two extremal cases: A tree $T$ in the
boundary of Outer space with dense orbits is of \textbf{surface type}
if $T\subseteq\Omega$ and $T$ is of \textbf{Levitt type} if $\Omega$
is totally disconnected. As the terminology suggests, a surface tree is
of surface type. Trees of Levitt type where discovered by
Levitt~\cite{levi-pseudogroup}.

Combining together the two sets of properties, we introduced in~\cite{ch-a} the following definitions.
A tree $T$ in $\partial\CVN$ with dense orbits is 
\begin{itemize}
\item a \textbf{surface tree} if it is both geometric and of surface type;
\item \textbf{Levitt} if it is geometric and of Levitt type;
\item \textbf{pseudo-surface} if it is not geometric and of surface type;
\item \textbf{pseudo-Levitt} if it is not geometric and of Levitt type
\end{itemize}

The following Theorem is the main result of this paper.

\begin{thmbold}
  Let $\Phi$ be an \iwip\ outer automorphism of $\FN$. Let $T_\Phi$
  and $T_{\Phi\inv}$ be its attracting and repelling trees. Then
  exactly one of the following occurs
\begin{enumerate} 
\item The trees $T_\Phi$ and $T_{\Phi\inv}$ are surface
  trees. Equivalently $\Phi$ is geometric.
\item The tree $T_\Phi$ is Levitt (i.e. geometric and of Levitt type),
  and the tree $T_{\Phi\inv}$ is pseudo-surface (i.e. non-geometric
  and of surface type). Equivalently $\Phi$ is parageometric.
\item The tree $T_{\Phi\inv}$ is Levitt (i.e. geometric and of Levitt
  type), and the tree $T_{\Phi}$ is pseudo-surface (i.e. non-geometric
  and of surface type). Equivalently $\Phi\inv$ is parageometric.
\item The trees $T_\Phi$ and $T_{\Phi\inv}$ are pseudo-Levitt
  (non-geometric and of Levitt type).
\end{enumerate}
Case (1) corresponds to toroidal iwips whereas cases (2), (3) and (4)
corresponds to atoroidal iwips. In case (4) the automorphism $\Phi$ is
called \textbf{pseudo-Levitt}.
\end{thmbold}

Gaboriau, Jaeger, Levitt and Lustig \cite{gjll} introduced the notion
of an \textbf{index} $\indaut(\Phi)$, computed from the
rank of the fixed subgroup and from the number of attracting fixed
points of the automorphisms $\phi$ in the outer class $\Phi$. Another
index for a tree $T$ in $\barCVN$ has been defined and studied by
Gaboriau and Levitt~\cite{gl-rank}, we call it the \textbf{geometric
  index} $\igeo(T)$. Finally in our paper~\cite{ch-a} we introduced
and studied the \textbf{$\CQ$-index} $\iq(T)$ of an $\R$-tree $T$ in
the boundary of outer space with dense orbits. The two indices
$\igeo(T)$ and $\iq(T)$ describe qualitative properties of the tree
$T$ \cite{ch-a}. We define these indices and recall our botanical
classification of trees in Section~\ref{sec:botanictrees}.

The key to prove Theorem~\ref{thm:boldconjecture} is:

\begin{propsind} 
Let $\Phi$ be an \iwip\ outer automorphism of $\FN$. Let $T_\Phi$ and
$T_{\Phi\inv}$ be its attracting and repelling trees. Replacing $\Phi$
by a suitable power, we have \[
2\indaut(\Phi)=\igeo(T_\Phi)=\iq(T_{\Phi\inv}). \]
\end{propsind}

We prove this Proposition in Sections~\ref{subsec:attracting} and
\ref{sec:iautoiQ}.

To study limit trees of \iwip\ automorphisms, we need to state that
they have the strongest mixing dynamical property, which is called
\textbf{indecomposability}.

\begin{thmiwipindec}
Let $\Phi\in\Out(\FN)$ be an iwip outer automorphism.
The attracting tree $T_\Phi$ of $\Phi$ is indecomposable.
\end{thmiwipindec}

The proof of this Theorem is quite independent of the rest of the
paper and is the purpose of Section~\ref{sec:app}. The proof relies on
a key property of iwip automorphisms: they can be represented by
(absolute) train-track maps.

\section{Indecomposability of the attracting tree of an \iwip\ automorphism}\label{sec:app}

Following Guirardel \cite{guir-indecomposable}, a (projective class of) $\R$-tree
$T\in\barCVN$ is {\bf indecomposable} if for all non degenerate arcs
$I$ and $J$ in $T$, there exists finitely many elements $u_1,\dots,u_n$ in $\FN$ such that
\begin{equation}\label{eq:indecomposableunion}
        J\subseteq \bigcup_{i=1}^n u_i I
\end{equation}
and
\begin{equation}\label{eq:indecomposableintersection}
\forall i=1,\dots,n-1,\quad    u_i I\cap u_{i+1}I \text{ is a non degenerate arc.}
\end{equation}

The main purpose of this Section is to prove

\begin{thm}\label{thm:iwipindec}
Let $\Phi\in\Out(\FN)$ be an iwip outer automorphism.
The attracting tree $T_\Phi$ of $\Phi$ is indecomposable.
\end{thm}

Before proving this Theorem in Section~\ref{subsec:indecomposable}, we
collect the results we need from \cite{bh-traintrack} and \cite{gjll}.

\subsection{Train-track representative of $\Phi$}

The rose $R_N$ is the graph with one vertex $\ast$ and $N$ edges.  Its
fundamental group $\pi_1(R_N,\ast)$ is naturally identified with the
free group $F_N$. A {\bf marked graph} is a finite graph $G$ with a
homotopy equivalence $\tau: R_N\rightarrow G$.  The marking $\tau$
induces an isomorphism $\tau_\ast:\FN=\pi_1(R_N,\ast)\stackrel{\cong}{\to}\pi_1(G,v_0)$,
where $v_0=\tau(\ast)$.

A homotopy equivalence $f:G\rightarrow G$ defines an outer
automorphism of $F_N$. Indeed, if a path $m$ from $v_0$ to $f(v_0)$
is given, $a\mapsto mf(a)m^{-1}$ induces an automorphism $\phi$ of
$\pi_1(G,v_0)$, and thus of $\FN$ through the marking.  Another path
$m'$ from $v_0$ to $f(v_0)$ gives rise to another automorphism
$\phi'$ of $\FN$ in the same outer class $\Phi$.

A {\bf topological representative} of $\Phi\in\Out(\FN)$ is an
homotopy equivalence $f:G\rightarrow G$ of a marked graph $G$, such
that:
\begin{enumerate}\renewcommand{\labelenumi}{(\roman{enumi})}
\item $f$ maps vertices to vertices,
\item $f$ is locally injective on any edge,
\item $f$ induces $\Phi$ on $F_N\cong \pi_1(G,v_0)$.
\end{enumerate}

Let $e_1,\dots,e_p$ be the edges of $G$ (an orientation is arbitrarily
given on each edge, and $e^{-1}$ denotes the edge $e$ with the reverse
orientation).  The {\bf transition matrix} of the map $f$ is the
$p\times p$ non-negative matrix $M$ with $(i,j)$-entry equal to the
number of times the edge $e_i$ occurs in $f(e_j)$ (we say that a
path (or an edge) $w$ of a graph $G$ \textbf{occurs} in a path $u$ of $G$ if it is
$w$ or its inverse $w\inv$ is a subpath of $u$).

A topological representative  $f:G\rightarrow G$ of $\Phi$ is a
{\bf train-track map} if moreover:
\begin{enumerate}\setcounter{enumi}{3}\renewcommand{\labelenumi}{(\roman{enumi})}
\item for all $k\in\N$, the restriction of $f^k$ on any edge of $G$ is locally injective,
\item any vertex of $G$ has valence at least 3.
\end{enumerate}

According to \cite[Theorem 1.7]{bh-traintrack}, an \iwip\ outer
automorphism $\Phi$ can be represented by a train-track map, with a
primitive transition matrix $M$ (i.e. there exists some $k\in\N$ such
all the entries of $M^k$ are strictly positive). Thus the
Perron-Frobenius Theorem applies. In particular, $M$ has a real
dominant eigenvalue $\lambda>1$ associated to a strictly positive
eigenvector $u=(u_1,\dots,u_p)$.  Indeed, $\lambda$ is the expansion
factor of $\Phi$: $\lambda=\lambda_\Phi$.  We turn the graph $G$ to a
metric space by assigning the length $u_i$ to the edge $e_i$ (for
$i=1,\dots,p$).  Since, with respect to this metric, the length of
$f(e_i)$ is $\lambda$ times the length of $e_i$, we can assume that,
on each edge, $f$ is linear of ratio $\lambda$.

We define the set $\mathcal{L}_2(f)$ of paths $w$ of combinatorial
length 2 (i.e. $w=ee'$, where $e$, $e'$ are edges of $G$, $e\inv\neq
e'$) which occurs in some $f^k(e_i)$ for some
$k\in\N$ and some edge $e_i$ of $G$:
\[
\mathcal{L}_2(f) = \{
ee' : \exists e_i \text{ edge of } G,\;
\exists k\in\N \text{ such that }
ee' \text{ is a subpath of } f^k({e_i}^{\pm 1})
\}.
\]
Since the transition matrix $M$ is primitive,
there exists $k\in\N$ such that for any edge
$e$ of $G$, for any $w\in\mathcal{L}_2(f)$,
$w$ occurs in $f^k(e)$.

Let $v$ be a vertex of $G$. The {\bf Whitehead graph}
$\mathcal{W}_v$ of $v$ is the unoriented graph
defined by:
\begin{itemize}
\item the vertices of $\mathcal{W}_v$ are the
edges of $G$ with $v$ as terminal vertex,
\item there is an edge in $\mathcal{W}_v$ between
$e$ and $e'$ if $e'e^{-1}\in\mathcal{L}_2(f)$.
\end{itemize}
As remarked in \cite[Section~2]{bfh-lam}, if 
$f:G\rightarrow G$ is a train-track representative
of an \iwip\ outer automorphism $\Phi$,
any vertex of $G$ has a connected Whitehead graph.
We summarize the previous discussion in:
\begin{prop}\label{prop:train-track}
Let $\Phi\in\Out(\FN)$ be an iwip outer automorphism.  There exists a
train-track representative $f:G\rightarrow G$ of $\Phi$, with
primitive transition matrix $M$ and connected Whitehead graphs of
vertices.  The edge $e_i$ of $G$ is isometric to the segment
$[0,u_i]$, where $u=(u_1,\dots,u_p)$ is a Perron-Frobenius
eigenvector of $M$. The map $f$ is linear of ratio
$\lambda$ on each edge $e_i$ of $G$.
\end{prop}

\begin{rem}\label{rem:whitehead}
Let $f:G\rightarrow G$ be a train-track map, 
with primitive transition matrix $M$ and connected
Whitehead graphs of vertices.
Then for any path $w=ab$ in $G$ of combinatorial length 2,
there exist $w_1=a_1b_1,\dots,w_q=a_qb_q\in
\mathcal{L}_2(f)$ ($a,b,a_i,b_i$ edges of $G$)
such that:
\begin{itemize}
\item $a_{i+1}=b_{i}^{-1}$, $i\in\{1,\dots, q-1\}$
\item $a=a_1$ and $b=b_q$.
\end{itemize}
\end{rem}

\subsection{Construction of $T_\Phi$}\label{sec:Tphi}

Let $\Phi\in\Out(\FN)$ be an iwip automorphism,
and let $T_\Phi$ be its attracting tree.
Following \cite{gjll}, we recall 
a concrete construction of the tree $T_\Phi$.

We start with a train-track representative $f:G\rightarrow G$
of $\Phi$ as in Proposition~\ref{prop:train-track}. 
The universal cover $\tilde G$ of $G$ is a simplicial tree,
equipped with a distance $d_0$ obtained by lifting
the distance on $G$. 
The fundamental group $\FN$ acts by deck transformations,
and thus by isometries, on $\tilde G$.
Let $\tilde f$ be a lift of $f$ to $\tilde G$.
This lift $\tilde f$ is associated to a unique automorphism $\phi$ in the outer
class $\Phi$, characterized by 
\begin{equation}\label{eq:phi-f}
\forall u\in\FN, \forall x\in\tilde G,\quad \phi(u)\tilde f(x)=\tilde f(ux).
\end{equation}

For $x,y\in\tilde G$ and $k\in\N$, we define:
\[
d_k(x,y)=\frac{d_0(\tilde f^k(x),\tilde f^k(y))}{\lambda^k}.
\]
The sequence of distances $d_k$ is decreasing and converges to a
pseudo-distance $d_\infty$ on $\tilde G$.  Identifying points $x,y$ in
$\tilde G$ which have distance $d_\infty(x,y)$ equal to $0$, we obtain
the tree $T_\Phi$.  The free group $\FN$ still acts by isometries on
$T_\Phi$.  The quotient map $p:\tilde G\rightarrow T_\Phi$ is
$\FN$-equivariant and 1-Lipschitz.  Moreover, for any edge $e$ of
$\tilde G$, for any $k\in\N$, the restriction of $p$ to $f^k(e)$ is an
isometry. Through $p$ the map $\tilde f$ factors to a homothety $H$
of $T_\Phi$, of ratio $\lambda_\Phi$:
\[
\forall x\in\tilde G,\quad H(p(x))=p(\tilde f(x)).
\]  
Property
(\ref{eq:phi-f}) leads to 
\begin{equation}\label{eq:H}
\forall u\in\FN, \forall x\in T_\Phi, \quad \phi(u)H(x)=H(ux).
\end{equation}

\subsection{Indecomposability of $T_\Phi$}
\label{subsec:indecomposable}

We say that a path (or an edge) $w$ of the graph $G$ occurs in a path
$u$ of the universal cover $\tilde G$ of $G$ if $w$ has a lift $\tilde
w$ which occurs in $u$.

\begin{lem}\label{lem:Ilift}
Let $I$ be a non degenerate arc in $T_\Phi$.
There exists an arc $I'$ in $\tilde G$ and an integer
$k$ such that
\begin{itemize}
\item $p(I')\subseteq I$
\item any element of $\mathcal{L}_2(f)$ occurs in
$H^k(I')$.
\end{itemize}
\end{lem}
\begin{proof}
Let $I\subset T_\Phi$ be a non-degenerate arc.
There exists an edge $e$ of $\tilde G$ such that
$I_0=p(e)\cap I$ is a non-degenerate arc: $I_0=[x,y]$.
We choose $k_1\in\N$ such that 
$d_\infty(H^{k_1}(x), H^{k_1}(y)) > L$
where
\[
L=2\max\{u_i=|e_i| \ |\ e_i \text{ edge of } G \}.
\]
Let $x',y'$ be the points in $e$ such that $p(x')=x$, $p(y')=y$, and
let $I'$ be the arc $[x',y']$.  Since $p$ maps $f^{k_1}(e)$
isometrically into $T_\Phi$, we obtain that
$d_0(f^{k_1}(x'),f^{k_1}(y')) \geq L$ Hence there exists an edge $e'$
of $\tilde G$ contained in $[f^{k_1}(x'),f^{k_1}(y')]$.  Moreover, for
any $k_2\in\N$, the path $f^{k_2}(e')$ isometrically injects in
$[H^{k_1+k_2}(x),H^{k_1+k_2}(y)]$.  We take $k_2$ big enough so that
any path in $\mathcal{L}_2(f)$ occurs in $f^{k_2}(e')$.  Then
$k=k_1+k_2$ is suitable.
\end{proof}

\begin{proof}[Proof of Theorem~\ref{thm:iwipindec}]  Let $I,J$ be two
non-trivial arcs in $T_\Phi$.  We have to prove that $I$ and $J$
satisfy properties~(\ref{eq:indecomposableunion}) and
(\ref{eq:indecomposableintersection}).  Since $H$ is a homeomorphism,
and because of (\ref{eq:H}), we can replace $I$ and $J$ by $H^k(I)$
and $H^k(J)$, accordingly, for some $k\in\N$.

We consider an arc $I'$ in $\tilde G$ and an integer $k\in\N$ as given
by Lemma \ref{lem:Ilift}.  Let $x,y$ be the endpoints of the arc
$H^k(J)$: $H^k(J)=[x,y]$.  Let $x',y'$ be points in $\tilde G$ such
that $p(x')=x$, $p(y')=y$, and let $J'$ be the arc $[x',y']$.
According to Remark \ref{rem:whitehead}, there exist $w_1,\dots,w_n$
such that:
\begin{itemize}
\item $w_i$ is a lift of some path in $\mathcal{L}_2(f)$,
\item $J'\subseteq \bigcup_{i=1}^n w_i$, 
\item $w_i\cap w_{i+1}$ is an edge.
\end{itemize}
Since Lemma \ref{lem:Ilift} ensures that any element of
$\mathcal{L}_2(f)$ occurs in $H^k(I')$, we deduce that $H^k(I)$ and
$H^k(J)$ satisfy properties~(\ref{eq:indecomposableunion}) and
(\ref{eq:indecomposableintersection}),
concluding the proof of Theorem~\ref{thm:iwipindec}. 
\end{proof}

\section{Index of an outer automorphism}\label{sec:autoindex}

An automorphism $\phi$ of the free group $\FN$ extends to a
homeomorphism $\partial\phi$ of the boundary at infinity
$\partial\FN$. We denote by $\Fix(\phi)$ the fixed subgroup of
$\phi$. It is a finitely generated subgroup of $\FN$ and thus
its boundary $\partial\Fix(\phi)$ naturally embeds in
$\partial\FN$. Elements of $\partial\Fix(\phi)$ are fixed by
$\partial\phi$ and they are called \textbf{singular}.
Non-singular fixed points of $\partial\phi$ are called
\textbf{regular}. A fixed point $X$ of $\partial\phi$ is
\textbf{attracting} (resp. \textbf{repelling}) if it is regular
and if there exists an element $u$ in $\FN$ such that
$\phi^n(u)$ (resp. $\phi^{-n}(u)$) converges to $X$.
The set of fixed points of $\partial\phi$ is denoted by 
$\Fix(\partial\phi)$.

Following Nielsen, fixed points of $\partial\phi$ have been
classified by Gaboriau, Jaeger, Levitt and, Lustig:
\begin{prop}[{\cite[Proposition~1.1]{gjll}}] 
  Let $\phi$ be an automorphism of the free group $\FN$. Let $X$ be a
  fixed point of $\partial\phi$. Then exactly one of the following
  occurs:
\begin{enumerate} 
\item $X$ is in the boundary of the fixed
subgroup of $\phi$; 
\item $X$ is attracting; 
\item $X$ is repelling.\qed
\end{enumerate}
\end{prop}

We denote by $\Att(\phi)$ the set of attracting fixed points of
$\partial\phi$. The fixed subgroup $\Fix(\phi)$ acts on the set
$\Att(\phi)$ of attracting fixed points.

In \cite{gjll} the following \textbf{index} of the automorphism
$\phi$ is defined:
\[
\indaut(\phi)=\frac{1}{2}\#(\Att(\phi)/\Fix(\phi))+\rank(\Fix(\phi))-1
\]
If $\phi$ has a trivial fixed subgroup, the above definition
is simpler:
\[
\indaut(\phi)=\frac 12\#\Att(\phi)-1.
\]

Let $u$ be an element of $\FN$ and let $i_u$ be the
corresponding inner automorphism of $\FN$:
\[
\forall w\in\FN,
i_u(w)=uwu\inv.
\]
The inner automorphism $i_u$ extends to the boundary of $\FN$ as left
multiplication by $u$:
\[
\forall
X\in\partial\FN, \partial i_u(X)=uX.
\]
The group $\Inn(\FN)$ of
inner automorphisms of $\FN$ acts by conjugacy on the
automorphisms in an outer class $\Phi$. Following Nielsen, two
automorphisms, $\phi,\phi'\in\Phi$ are \textbf{isogredient} if
they are conjugated by some inner automorphism $i_u$:
\[
\phi'=i_u\circ\phi\circ i_{u\inv}=i_{u\phi(u)\inv}\circ\phi.
\]
In this case, the actions of $\partial\phi$ and $\partial\phi'$
on $\partial\FN$ are conjugate by the left multiplication by
$u$. In particular, a fixed point $X'$ of $\partial\phi'$ is a
translate $X'=uX$ of a fixed point $X$ of $\partial\phi$. Two
isogredient automorphisms have the same index: this is the
index of the isogrediency class.
An isogrediency class $[\phi]$ is \textbf{essential} if it has positive
index: $\ind([\phi])>0$. We note that essential isogrediency
classes are principal in the sense of \cite{fh-recognition}, but the
converse is not true.

The \textbf{index} of the outer automorphism $\Phi$ is the sum,
over all essential isogrediency classes of automorphisms $\phi$ in the
outer class $\Phi$, of their indices, or alternatively:
\[
\indaut(\Phi)=\sum_{[\phi]\in\Phi/\mbox{\scriptsize
Inn}(\FN)}\max(0;\indaut(\phi)).
\]

We adapt the notion of {\em forward rotationless outer automorphism}
of Feighn and Handel \cite{fh-recognition} to our purpose.  We denote
by $\Per(\phi)$ the set of elements of $\FN$ fixed by some positive
power of $\phi$:
\[
\Per(\phi)=\bigcup_{n\in\N^*}\Fix(\phi^n);
\]
and by $\Per(\partial\phi)$ the set of elements of $\partial\FN$ fixed
by some positive power of $\partial\phi$:
\[
\Per(\partial\phi)=\bigcup_{n\in\N^*}\Fix(\partial\phi^n).
\]

\begin{defn}
An outer automorphism $\Phi\in\Out(\FN)$ is \FR\ if:
\begin{enumerate}[(FR1)]
\item \label{fr1} for any automorphism $\phi\in\Phi$, $\Per(\phi)=\Fix(\phi)$ 
and $\Per(\partial\phi)=\Fix(\partial\phi)$;
\item \label{fr2} if $\psi$ is an automorphism
in the outer class $\Phi^n$ for some $n>0$, with
$\indaut(\psi)$ postive, then there exists an automorphism $\phi$ in
$\Phi$ such that $\psi=\phi^n$. 
\end{enumerate}
\end{defn}

\begin{prop}\label{prop:rotationless}
Let $\Phi\in\Out(\FN)$.
There exists $k\in\N^*$ such that $\Phi^k$ is \FR.
\end{prop}

\begin{proof}
By \cite[ Theorem~1]{ll-periodic-ends} there exists a power
$\Phi^k$ with (FR\ref{fr1}).  An automorphism $\phi\in\Aut(\FN)$ with
positive index $\ind(\phi)>0$ is principal in the sense of
\cite[Definition 3.1]{fh-recognition}.  Thus our property
(FR\ref{fr2}) is a consequence of the forward rotationless property of
\cite[Definition 3.13]{fh-recognition}. By \cite[Lemma
4.43]{fh-recognition} there exists a power $\Phi^{k\ell}$ which is
forward rotationless and thus which satisfies (FR\ref{fr2}).
\end{proof}

\section{Indices}

\subsection{Botany of trees}\label{sec:botanictrees}

We recall in this Section the classification of trees in the boundary
of outer space of our paper~\cite{ch-a}.

Gaboriau and Levitt~\cite{gl-rank} introduced an index for a tree $T$
in $\barCVN$ , we call it the \textbf{geometric index} and denote it by
$\igeo(T)$. It is defined using the valence of the branch points, of
the $\R$-tree $T$, with an action of the free group by isometries:
\[ 
\igeo(T)=\sum_{[P]\in T/\FN}\igeo(P). 
\]
where the local index of a point $P$ in $T$ is
\[
\igeo(P)=\#(\pi_0(T\smallsetminus\{P\})/\Stab(P))+2\,\rank(\Stab(P))-2.
\]
Gaboriau and Levitt \cite{gl-rank} proved that the geometric index of
a geometric tree is equal to $2N-2$ and that for any tree in the
compactification of outer space $\barCVN$ the geometric index is
bounded above by $2N-2$. Moreover, they proved that the trees in
$\barCVN$ with geometric index equal to $2N-2$ are precisely
the geometric trees.

If, moreover, $T$ has dense orbits, Levitt and Lustig \cite{ll-north-south, ll-periodic}
defined the map $\CQ:\partial\FN\to\hat T$ which is characterized
by

\begin{prop}\label{prop:Qexists} 
  Let $T$ be an $\R$-tree in $\barCVN$ with dense orbits.  There
  exists a unique map $\CQ:\partial\FN\to\hat T$ such that for any
  sequence $(u_n)_{n\in\N}$ of elements of $\FN$ which converges to
    $X\in\partial\FN$, and any point $P\in T$, if the sequence of
    points $(u_nP)_{n\in\N}$ converges to a point $Q\in\hat T$, then
    $\CQ(X)=Q$. Moreover, $\CQ$ is onto.
\end{prop}

Let us consider the case of a tree $T$ dual to a measured foliation
$(\mathcal{F},\mu)$ on a hyperbolic surface $S$ with boundary ($T$ is
a surface tree). Let $\tilde{\mathcal{F}}$ be the lift of
$\mathcal{F}$ to the universal cover $\tilde S$ of $S$. The boundary
at infinity of $\tilde S$ is homeomorphic to $\partial\FN$. On the one
hand, a leaf $\ell$ of $\tilde{\mathcal{F}}$ defines a point in
$T$. On the other hand, the ends of $\ell$ define points in
$\partial\FN$. The map $\CQ$ precisely sends the ends of $\ell$ to the
point in $T$. The Poincaré-Lefschetz index of the foliation
$\mathcal{F}$ can be computed from the cardinal of the fibers of the
map $\CQ$. This leads to the following definition of the $\CQ$-index
of an $\R$-tree $T$ in a more general context.

Let $T$ be an $\R$-tree in $\barCVN$ with dense orbits.  The
$\CQ$-index of the tree $T$ is defined as follows:
\[ 
\iq(T)=\sum_{[P]\in \hat T/\FN}\max(0;\iq(P)). 
\] 
where the local index of a point $P$ in $T$ is:
\[ 
\iq(P)=\#(\CQ_r\inv(P)/\Stab(P))+2\,\rank(\Stab(P))-2 
\]
with $\CQ_r\inv(P)=\CQ\inv(P)\ssm\partial\Stab(P)$ the regular fiber
of $P$.

Levitt and Lustig~\cite{ll-north-south} proved that points in
$\partial T$ have exactly one pre-image by $\CQ$.  Thus, only points
in $\bar T$ contribute to the $\CQ$-index of $T$.

We proved \cite{ch-a} that the $\CQ$-index of an $\R$-tree in the
boundary of outer space with dense orbits is bounded above by
$2N-2$. And it is equal to $2N-2$ if and only if it is of surface
type.

Our botanical classification \cite{ch-a} of a tree $T$ with a
minimal very small indecomposable action of $\FN$ by isometries is as
follows

\begin{center}
\tabcolsep=.2em \noindent
\begin{tabular}{|c|c||c|c|} 
\cline{3-4}
\multicolumn{2}{c||}{} & \textbf{geometric} & \textbf{not geometric}\\ 
\cline{3-4}
\multicolumn{2}{c||}{} &
	\begin{tabular}{c}	
	$\igeo(T)=2N-2$
	\end{tabular} & 
		\begin{tabular}{c}
		$\igeo(T)<2N-2$
		\end{tabular}\\
\hline\hline
\textbf{Surface type}& 
		\begin{tabular}{c}
		$\iq(T)=2N-2$
		\end{tabular} &
			\textbf{surface} & \textbf{pseudo-surface}\\
\hline
\textbf{Levitt type}& 
		\begin{tabular}{c}
		$\iq(T)<2N-2$
		\end{tabular} 
			& \textbf{Levitt} & \textbf{pseudo-Levitt}\\
\hline 
\end{tabular}
\end{center}

The following remark is not necessary for the sequel of the paper, but
may help the reader's intuition.

\noindent\textbf{Remark.} In~\cite{chl1-I,chl1-II}, in collaboration with
Lustig, we defined and studied the dual
lamination of an $\R$-tree $T$ with dense orbits:
\[
L(T)=\{(X,Y)\in\partial^2\FN\ |\ \CQ(X)=\CQ(Y)\}.
\]
The $\CQ$-index of $T$ can be interpreted as the index of this dual
lamination.  

Using the dual lamination, with Lustig~\cite{chl4}, we defined the
compact heart $K_A\subseteq\bar T$ (for a basis $A$ of $\FN$). We
proved that the tree $T$ is completely encoded by a system of partial
isometries $S_A=(K_A,A)$. We also proved that the tree $T$ is
geometric if and only if the compact heart $K_A$ is a finite tree
(that is to say the convex hull of finitely many points). In our
previous work~\cite{ch-a} we used the Rips machine on the system of
isometries $S_A$ to get the bound on the $\CQ$-index of $T$. In
particular, an indecomposable tree $T$ is of Levitt type if and only
if the Rips machine never halts.

\subsection{Geometric index}\label{subsec:attracting}

As in Section~\ref{sec:Tphi}, an \iwip\ outer automorphism $\Phi$ has
an expansion factor $\lambda_\Phi>1$, an attracting $\R$-tree $T_\Phi$
in $\partial\CVN$. For each automorphism $\phi$ in the outer
class $\Phi$ there is a homothety $H$ of the metric completion $\bar
T_\Phi$, of ratio $\lambda_\Phi$, such that
\begin{equation}\label{eqn:Hphi} \forall P\in \bar T_\Phi,\;\;
\forall u\in\FN,\;\;\; H(uP)=\phi(u)H(P)
\end{equation} 

In addition, the action of $\Phi$ on the compactification of Culler
and Vogtmann's Outer space has a North-South dynamic and the
projective class of $T_\Phi$ is the attracting fixed point
\cite{ll-north-south}. Of course the attracting trees of $\Phi$ and
$\Phi^n$ ($n>0$) are equal.

For the attracting tree $T_\Phi$ of the \iwip\ outer automorphism
$\Phi$, the geometric index is well understood.

\begin{prop}[{\cite[Section~4]{gjll}}]\label{prop:geomauto} 
Let $\Psi$ be an \iwip\ outer automorphism. 
There exists a power $\Phi=\Psi^k$ ($k>0$) of $\Psi$ such
that:

\[
2\indaut(\Phi)=\igeo(T_\Phi),
\]
where $T_\Phi$ is the attracting tree of $\Phi$ (and of $\Psi$). \qed
\end{prop}

\subsection{{$\CQ$}-index}\label{sec:iautoiQ}

Let $\Phi$ be an \iwip\ outer automorphism of $\FN$. Let $T_\Phi$ be
its attracting tree.  The action of $\FN$ on $T_\Phi$ has dense
orbits. 

Let $\phi$ an automorphism in the outer class $\Phi$.
The homothety $H$ associated to $\phi$ 
extends continuously to an homeomorphism of
the boundary at infinity of $T_\Phi$ which we still denote by $H$. We
get from Proposition~\ref{prop:Qexists} and
identity~\ref{eqn:Hphi}:
\begin{equation}\label{eqn:QH} \forall
X\in\partial\FN, \CQ(\partial\phi(X))=H(\CQ(X)).
\end{equation}

We are going to prove that the $\CQ$-index of $T_{\Phi}$ is
twice the index of $\Phi\inv$. As mentioned in the introduction
for geometric automorphisms both these numbers are equal to
$2N-2$ and thus we restrict to the study of non-geometric
automorphisms. For the rest of this section we assume that
$\Phi$ is non-geometric. 
This will be used in two ways:
\begin{itemize}
\item the action of $\FN$ on $T_\Phi$ is free;
\item for any $\phi$ in the outer class $\Phi$, all the fixed points
of $\phi$ in $\partial\FN$ are regular.
\end{itemize}

Let $C_H$ be the center of the homothety $H$. The following
Lemma is essentially contained in \cite{gjll}, although
the map $\CQ$ is not used there.

\begin{lem}\label{lem:Qindexphiindexlocal} 
Let $\Phi\in\Out(\FN)$ be a \FR\ non-geometric \iwip\ outer 
automorphism. 
Let $T_\Phi$ be the attracting tree of $\Phi$.
Let $\phi\in\Phi$ be an automorphism in the outer class $\Phi$, 
and let $H$ be the homothety of $T_\Phi$
associated to $\phi$, with $C_H$ its center.
The $\CQ$-fiber of $C_H$ is the set of repelling points of
$\phi$. 
\end{lem} 

\begin{proof} 
Let $X\in\partial\FN$ be a
repelling point of $\partial\phi$. By definition there exists an
element $u\in\FN$ such that the sequence $(\phi^{-n}(u))_n$
converges towards $X$. By Equation~\ref{eqn:Hphi},
\[
\phi^{-n}(u)C_H=\phi^{-n}(u)H^{-n}(C_H)=H^{-n}(uC_H).
\]
The
homothety $H^{-1}$ is strictly contracting and thus the sequence
of points $(\phi^{-n}(u)C_H)_n$ converges towards $C_H$. By
Proposition~\ref{prop:Qexists} we get that $\CQ(X)=C_H$.

Conversely let $X\in\CQ\inv(C_H)$ be a point in the $\CQ$-fiber of
$C_H$. Using the identity~\ref{eqn:QH}, $\partial\phi(X)$ is also in
the $\CQ$-fiber. The $\CQ$-fiber is finite by
\cite[Corollary~5.4]{ch-a}, $X$ is a periodic point of   
$\partial\phi$.  Since $\Phi$ satisfies property (FR\ref{fr1}), $X$ is
a fixed point of $\partial\phi$. From \cite[Lemma~3.5]{gjll},
attracting fixed points of $\partial\phi$ are mapped by $\CQ$ to
points in the boundary at infinity $\partial T_\Phi$. Thus $X$ has to
be a repelling fixed point of $\partial\phi$.
\end{proof}

\begin{prop}\label{prop:indexrank} 
Let $\Phi\in\Out(\FN)$ be a \FR\ non-geometric \iwip\ outer 
automorphism. Let $T_\Phi$ be the attracting tree of $\Phi$.
Then
\[
2\indaut(\Phi\inv)=\iq(T_{\Phi}).\]
\end{prop}

\begin{proof} 
To each automorphism $\phi$ in the outer class $\Phi$ is associated a
homothety $H$ of $T_\Phi$ and the center $C_H$ of this homothety.  As
the action of $\FN$ on $T_\Phi$ is free, two automorphisms are
isogredient if and only if the corresponding centers are in the same
$\FN$-orbit.

The index of $\Phi\inv$ is the sum over all essential 
isogrediency classes
of automorphism $\phi\inv$ in $\Phi\inv$ of the index of
$\phi\inv$. For each of these automorphisms the index
$2\indaut(\phi\inv)$ is equal by
Proposition~\ref{lem:Qindexphiindexlocal} to the contribution
$\#\CQ^{-1}(C_H)$
of the orbit of $C_H$ to the $\CQ$ index of $T_\Phi$.

Conversely, let now $P$ be a point in $\bar T_{\Phi}$ with at least
three elements in its $\CQ$-fiber. Let $\phi$ be an automorphism in
$\Phi$ and let $H$ be the homothety of $T_{\Phi}$ associated to
$\phi$. For any integer $n$, the $\CQ$-fiber
$\CQ\inv(H^n(P))=\partial\phi^{n}(\CQ\inv(P))$ of $H^n(P)$ also has at
least three elements. By \cite[Theorem~5.3]{ch-a} there are finitely many 
orbits of such points in $T_{\Phi}$ and thus we can assume that
$H^n(P)=wP$ for some $w\in\FN$ and some integer $n>0$. Then $P$ is the
center of the homothety $w\inv H^n$ associated to
$i_{w\inv}\circ\phi^{n}$. 
Since $\Phi$ satisfies property (FR\ref{fr2}),
$P$ is the center of a homothety $uH$ associated to
$i_u\circ\phi$ for some $u\in\FN$.
This concludes the proof of the equality of the indices.
\end{proof}

This Proposition can alternatively be deduced from the techniques of
Handel and Mosher~\cite{hm-axes}.

\section{Botanical classification of irreducible automorphisms}\label{sec:pairing}

\begin{thm}\label{thm:QindexTindex} 
Let $\Phi$ be an \iwip\  outer
automorphism of $\FN$. Let $T_\Phi$ and $T_{\Phi\inv}$ be its
attracting and repelling trees. Then, the $\CQ$-index of the
attracting tree is equal to the geometric index of the repelling tree:
\[
\iq(T_\Phi)=\igeo(T_{\Phi\inv}).\]

\end{thm}
\begin{proof} 
First, if $\Phi$ is geometric, then the trees $T_\Phi$
and $T_{\Phi\inv}$ have maximal geometric indices $2N-2$. On the
other hand the trees $T_\Phi$ and $T_{\Phi\inv}$ are surface trees and thus
their $\CQ$-indices are also maximal: 
\[
\igeo(T_{\Phi})=\iq(T_\Phi)=\igeo(T_{\Phi\inv})=\iq(T_{\Phi\inv})=2N-2
\]

We now assume that $\Phi$ is not geometric and we can apply
Propositions~\ref{prop:geomauto} and \ref{prop:indexrank} to get
the desired equality.
\end{proof}

From Theorem~\ref{thm:QindexTindex} and from the
characterization of geometric and surface-type trees by the
maximality of the indices we get

\begin{thm}\label{thm:boldconjecture} Let $\Phi$ be an
\iwip\  outer automorphism of
$\FN$. Let $T_\Phi$ and $T_{\Phi\inv}$ be its attracting and
repelling trees. Then exactly one of the following occurs
\begin{enumerate} 
\item\label{item:geometric} $T_\Phi$ and
$T_{\Phi\inv}$ are surface trees; 
\item\label{item:para} $T_\Phi$ is
Levitt and $T_{\Phi\inv}$ is pseudo-surface; 
\item\label{item:contra} $T_{\Phi\inv}$
is Levitt and $T_{\Phi}$ is pseudo-surface;
\item\label{item:levitt} $T_\Phi$ and $T_{\Phi\inv}$ are
pseudo-Levitt.
\end{enumerate}
\end{thm}
\begin{proof} 
  The trees $T_\Phi$ and $T_{\Phi\inv}$ are indecomposable by
  Theorem~\ref{thm:iwipindec} and thus they are either of surface type
  or of Levitt type by
  \cite[Proposition~5.14]{ch-a}.  
Recall, from \cite{gl-rank} (see also \cite[Theorem~5.9]{ch-a} or   
\cite[Corollary~6.1]{chl4}) that $T_\Phi$ is
geometric
if and only if its geometric index is maximal:
\[
\igeo(T_\Phi)=2N-2.
\]
From \cite[Theorem~5.10]{ch-a}, $T_\Phi$  
is of surface type if and only if its $\CQ$-index is maximal:
\[
\iq(T_\Phi)=2N-2.
\]
The Theorem now follows from
Theorem~\ref{thm:QindexTindex}.
\end{proof}

Let $\Phi\in\Out(F_N)$ be an \iwip\  outer automorphism.

The outer automorphism $\Phi$ is {\bf geometric} if both its
attracting and repelling trees $T_\Phi$ and $T_{\Phi^{-1}}$ are
geometric. This is equivalent to saying that $\Phi$ is induced by a
pseudo-Anosov homeomorphism of a surface with boundary, see
\cite{guir-core} and \cite{hm-parageometric}. This is
case~\ref{item:geometric} of Theorem~\ref{thm:boldconjecture}.

The outer automorphism $\Phi$ is {\bf parageometric} if its
attracting tree $T_\Phi$ is geometric but its repelling tree
$T_{\Phi^{-1}}$ is not. This is case~\ref{item:para} of
Theorem~\ref{thm:boldconjecture}.

The outer automorphism $\Phi$ is \textbf{pseudo-Levitt} if both
its attracting and repelling trees are not geometric. This is
case~\ref{item:levitt} of Theorem~\ref{thm:boldconjecture}

\medskip

We now bring expansion factors into play.  An \iwip\ outer
automorphism $\Phi$ of $\FN$ has an expansion factor $\lambda_\Phi>1$:
it is the exponential growth rate of (non fixed) conjugacy classes
under iteration of $\Phi$.

If $\Phi$ is geometric, the expansion factor of $\Phi$ is equal
to the expansion factor of the associated pseudo-Anosov mapping
class and thus $\lambda_\Phi=\lambda_{\Phi\inv}$.

Handel and Mosher \cite{hm-parageometric} proved that if $\Phi$ is a
parageometric outer automorphism of $\FN$ then
$\lambda_\Phi>\lambda_{\Phi^{-1}}$ (see also \cite{bbc}). Examples are also given by
Gautero~\cite{gaut-mappingtorus}.

For pseudo-Levitt outer automorphisms of $\FN$ nothing can be said on
the comparison of the expansion factors of the automorphism and its
inverse. On one hand, Handel and Mosher give in the introduction of
\cite{hm-parageometric} an explicit example of a non geometric
automorphism with $\lambda_\Phi=\lambda_{\Phi\inv}$: thus this
automorphism is pseudo-Levitt. On the other hand, there are examples
of pseudo-Levitt automorphisms with
$\lambda_\Phi>\lambda_{\Phi^{-1}}$.  Let $\phi\in\Aut(F_3)$ be the automorphism
such that 
\[
\phi:\begin{array}[t]{rcl}a&\mapsto&b\\ b&\mapsto&ac\\ c&\mapsto&a\end{array}
\quad\text{ and }\quad\phi^{-1}:\begin{array}[t]{rcl}a&\mapsto&c\\ b&\mapsto&a\\ c&\mapsto&c^{-1}b\end{array}
\]
Let $\Phi$ be its outer class. Then $\Phi^6$ is \FR, has index
$\ind(\Phi^6)=3/2<2$. The expansion factor is $\lambda_\Phi\simeq
1,3247$.  The outer automorphism $\Phi^{-3}$ is \FR, has index
$\ind(\Phi^{-3})=1/2<2$. The expansion factor is
$\lambda_{\Phi^{-1}}\simeq 1,4655 > \lambda_\Phi$.  The computation of
these two indices can be achieved using the algorithm of
\cite{jull-these}.

\medskip

Now that we have classified outer automorphisms of $\FN$ into four
categories, questions of genericity naturally arise.  In particular,
is a generic outer automorphism of $\FN$ \iwip, pseudo-Levitt and with
distinct expansion factors? This is suggested by Handel and Mosher
\cite{hm-parageometric}, in particular for statistical genericity:
given a set of generators of $\Out(F_N)$ and considering the
word-metric associated to it, is it the case that
\[
\lim_{k\to\infty}\frac{\#(\textrm{pseudo-Levitt iwip with }\lambda_\Phi\neq\lambda_{\Phi^{-1}})\cap B(k))}
{\#B(k)}=1
\]
where $B(k)$ is the ball of radius $k$, centered
at $1$, in $\Out(F_N)$?

\subsection{Botanical memo}

In this Section we give a glossary of our classification of  automorphisms for the working mathematician.
                
For a \FR\  \iwip\  outer automorphism $\Phi$ of $\FN$, we used $6$
indices which are related in the following way:

\[
\begin{array}{|c|} \hline
2\indaut(\Phi)=\igeo(T_\Phi)=\iq(T_{\Phi\inv})\\ \hline
2\indaut(\Phi\inv)=\igeo(T_{\Phi\inv})=\iq(T_{\Phi})\\ \hline
\end{array}
\]
All these indices are bounded above by $2N-2$.
We sum up our Theorem~\ref{thm:boldconjecture} in the following
table.

\begin{center}
\tabcolsep=.2em \noindent\begin{tabular}{|ccccc|} \hline
Automorphisms & & Trees & & Indices \\

\hline

$\Phi$ geometric & $\Leftrightarrow$ & $T_\Phi$ and
$T_{\Phi^{-1}}$ geometric & $\Leftrightarrow$ &
$\indaut(\Phi)=\indaut(\Phi^{-1})=N-1$\\

$\Updownarrow$&&$\Updownarrow$&&\\

$\Phi^{-1}$ geometric && $T_\Phi$ surface& &
\\

&&$\Updownarrow$&&\\

&& $T_{\Phi\inv}$ surface&&\\
 
\hline

$\Phi$ parageometric & $\Leftrightarrow$ & $\begin{cases} T_\Phi
\text{ geometric} \\ \mbox{and}\\ T_{\Phi^{-1}} \text{ non
geometric}
\end{cases}$ & $\Leftrightarrow$& $\begin{cases}
\indaut(\Phi)=N-1\\ \mbox{and}\\ \indaut(\Phi^{-1})<N-1
\end{cases}$\\
 
&&$\Updownarrow$&&\\

&& $T_\Phi$
Levitt & & \\

&&$\Updownarrow$&&\\

 & & $T_{\Phi^{-1}}$ pseudo-surface & &\\

\hline

$\Phi$ pseudo-Levitt & $\Leftrightarrow$ & $T_\Phi$ and
$T_{\Phi^{-1}}$ non geometric & & \\

$\Updownarrow$&&$\Updownarrow$&&\\

$\Phi^{-1}$ pseudo-Levitt && $T_\Phi$
pseudo-Levitt& $\Leftrightarrow$ & \raisebox{0pt}[0pt][0pt]{$\begin{cases}
\indaut(\Phi)<N-1\\ \mbox{and}\\ \indaut(\Phi\inv)<N-1
\end{cases}$}\\

&&$\Updownarrow$&&\\ &&$T_{\Phi\inv}$ pseudo-Levitt&&\\ \hline
\end{tabular}
\end{center}

\section*{Acknowledgments} 

We thank Martin Lustig for his constant interest in our work and the
referee for the suggested improvements.


\begin{thebibliography}{GJLL98}

\bibitem[BBC08]{bbc}
Jason Behrstock, Mladen Bestvina, and Matt Clay.
\newblock Growth of intersection numbers for free group automorphisms.
\newblock {\em J. Topol.}, 3:280-310 ,2010. 


\bibitem[Bes02]{best-survey}
Mladen Bestvina.
\newblock {$\mathbb R$}-trees in topology, geometry, and group theory.
\newblock In {\em Handbook of geometric topology}, pages 55--91. North-Holland,
  Amsterdam, 2002.

\bibitem[BF92]{bf-combi}
Mladen Bestvina and Mark Feighn.
\newblock A combination theorem for negatively curved groups.
\newblock {\em J. Differential Geom.}, 35(1):85--101, 1992.


\bibitem[BF95]{bf-stable}
Mladen Bestvina and Mark Feighn.
\newblock Stable actions of groups on real trees.
\newblock {\em Invent. Math.}, 121(2):287--321, 1995.

\bibitem[BFH97]{bfh-lam}
Mladen Bestvina, Mark Feighn and Michael Handel.
\newblock Laminations, trees, and irreducible automorphisms of free groups.
\newblock {\em Geom. Funct. Anal.}, 7(2):215--244, 1997.

\bibitem[BH92]{bh-traintrack}
Mladen Bestvina and Michael Handel.
\newblock Train tracks and automorphisms of free groups,
\newblock {\em Ann. of Math. (2)}, 135(2):1--51, 1992.

\bibitem[Bri00]{brink}
P.~Brinkmann.
\newblock Hyperbolic automorphisms of free groups.
\newblock {\em Geom. Funct. Anal.}, 10(5):1071--1089, 2000.


\bibitem[CH10]{ch-a}
Thierry Coulbois, and Arnaud Hilion.
\newblock Rips Induction: Index of the dual lamination of an {$\R$}-tree
\newblock Preprint. \arxivlink{1002.0972}.

\bibitem[CHL07]{chl2}
Thierry Coulbois, Arnaud Hilion, and Martin Lustig.
\newblock Non-unique ergodicity, observers' topology and the dual algebraic
  lamination for {$\mathbb{R}$}-trees.
\newblock {\em Illinois J. Math.}, 51(3):897--911, 2007.

\bibitem[CHL08a]{chl1-I}
Thierry Coulbois, Arnaud Hilion, and Martin Lustig.
\newblock {$\mathbb R$}-trees and laminations for free groups. {I}. {A}lgebraic
  laminations.
\newblock {\em J. Lond. Math. Soc. (2)}, 78(3):723--736, 2008.


\bibitem[CHL08]{chl1-II}
Thierry Coulbois, Arnaud Hilion, and Martin Lustig.
\newblock {$\mathbb R$}-trees and laminations for free groups. {II}. {T}he dual
  lamination of an {$\mathbb R$}-tree.
\newblock {\em J. Lond. Math. Soc. (2)}, 78(3):737--754, 2008.

\bibitem[CHL09]{chl4}
Thierry Coulbois, Arnaud Hilion, and Martin Lustig.
\newblock {$\mathbb R$}-trees, dual laminations, and compact systems of partial
  isometries.
\newblock {\em Math. Proc. Cambridge Phil. Soc.}, 147:345--368, 2009.

\bibitem[CL89]{cl-fixed}
Marshall~M. Cohen and Martin Lustig.
\newblock On the dynamics and the fixed subgroup of a free group automorphism.
\newblock {\em Invent. Math.}, 96(3):613--638, 1989.

\bibitem[CV86]{cv-moduli}
Marc Culler and Karen Vogtmann.
\newblock {Moduli of graphs and automorphisms of free groups.}
\newblock {\em Invent. Math.}, 84:91--119, 1986.

\bibitem[FH06]{fh-recognition}
Mark Feighn and Michael Handel.
\newblock The recognition theorem for {${\rm Out}(F\sb n)$}.
\newblock {\em Groups Geom. Dyn.}, 5:39-106, 2011.

\bibitem[FLP79]{flp}
A.~Fathi, F.~Laudenbach, and V.~Po\'{e}naru, editors.
\newblock {\em Travaux de {T}hurston sur les surfaces}.
\newblock Soci\'et\'e Math\'ematique de France, Paris, 1979.
\newblock Ast\'erisque No. 66-67.

\bibitem[Gab96]{gab-bouts}
Damien Gaboriau.
\newblock Dynamique des syst\`emes d'isom\'etries: sur les bouts des orbites.
\newblock {\em Invent. Math.}, 126(2):297--318, 1996.

\bibitem[Gab97]{gab-indgen}
Damien Gaboriau.
\newblock G\'en\'erateurs ind\'ependants pour les syst\`emes d'isom\'etries de
  dimension un.
\newblock {\em Ann. Inst. Fourier (Grenoble)}, 47(1):101--122, 1997.

\bibitem[Ger83]{gers}
S.~M. Gersten.
\newblock Intersections of finitely generated subgroups of free groups and
  resolutions of graphs.
\newblock {\em Invent. Math.}, 71(3):567--591, 1983.

\bibitem[GJLL98]{gjll}
Damien Gaboriau, Andre Jaeger, Gilbert Levitt, and Martin Lustig.
\newblock An index for counting fixed points of automorphisms of free groups.
\newblock {\em Duke Math. J.}, 93(3):425--452, 1998.

\bibitem[GL95]{gl-rank}
Damien Gaboriau and Gilbert Levitt.
\newblock The rank of actions on {${\mathbb R}$}-trees.
\newblock {\em Ann. Sci. \'Ecole Norm. Sup. (4)}, 28(5):549--570, 1995.

\bibitem[GLP94]{glp-rips}
Damien Gaboriau, Gilbert Levitt, and Frédéric Paulin.
\newblock Pseudogroups of isometries of {${\mathbb R}$} and {R}ips' theorem on
  free actions on {${\mathbb R}$}-trees.
\newblock {\em Israel J. Math.}, 87(1-3):403--428, 1994.

\bibitem[Gau07]{gaut-mappingtorus}
Fran{\c{c}}ois Gautero.
\newblock Combinatorial mapping-torus, branched surfaces and free group
  automorphisms.
\newblock {\em Ann. Sc. Norm. Super. Pisa Cl. Sci. (5)}, 6(3):405--440, 2007.

\bibitem[Gui05]{guir-core}
Vincent Guirardel.
\newblock C\oe ur et nombre d'intersection pour les actions de groupes sur les
  arbres.
\newblock {\em Ann. Sci. \'Ecole Norm. Sup. (4)}, 38(6):847--888, 2005.

\bibitem[Gui08]{guir-indecomposable}
Vincent Guirardel.
\newblock Actions of finitely generated groups on {$\mathbb R$}-trees.
\newblock {\em Ann. Inst. Fourier (Grenoble)}, 58(1):159--211, 2008.

\bibitem[HM06]{hm-axes}
Michael Handel and Lee Mosher.
\newblock Axes in outer space.
\newblock {\em Mem. Amer. Math. Soc.}, 1004, 2011.
\newblock \arxivlink{math/0605355}, 2006.

\bibitem[HM07]{hm-parageometric}
Michael Handel and Lee Mosher.
\newblock Parageometric outer automorphisms of free groups.
\newblock {\em Trans. Amer. Math. Soc.}, 359(7):3153--3183 (electronic), 2007.

\bibitem[Jul09]{jull-these}
Yann Jullian.
\newblock {\em Représentations géométriques des systèmes dynamiques
  substitutifs par substitutions d'arbre}.
\newblock PhD thesis, Universit\'e Aix-Marseille~II, 2009.

\bibitem[Kap01]{kapo-book}
Michael Kapovich.
\newblock {\em Hyperbolic manifolds and discrete groups}, volume 183 of {\em
  Progress in Mathematics}.
\newblock Birkh\"auser Boston Inc., Boston, MA, 2001.

\bibitem[Lev93]{levi-pseudogroup}
Gilbert Levitt.
\newblock La dynamique des pseudogroupes de rotations.
\newblock {\em Invent. Math.}, 113(3):633--670, 1993.

\bibitem[LL00]{ll-periodic-ends}
Gilbert Levitt and Martin Lustig.
\newblock Periodic ends, growth rates, {H}\"older dynamics for automorphisms of
  free groups.
\newblock {\em Comment. Math. Helv.}, 75(3):415--429, 2000.

\bibitem[LL03]{ll-north-south}
Gilbert Levitt and Martin Lustig.
\newblock Irreducible automorphisms of {$F\sb n$} have north-south dynamics on
  compactified outer space.
\newblock {\em J. Inst. Math. Jussieu}, 2(1):59--72, 2003.

\bibitem[LL08]{ll-periodic}
Gilbert Levitt and Martin Lustig.
\newblock Automorphisms of free groups have asymptotically periodic dynamics.
\newblock {\em J. Reine Angew. Math.}, 619:1--36, 2008.

\bibitem[Mor88]{morg-ergodic}
John~W. Morgan.
\newblock Ergodic theory and free actions of groups on {${\bf R}$}-trees.
\newblock {\em Invent. Math.}, 94(3):605--622, 1988.

\bibitem[Rey10]{reyn}
Patrick Reynolds.
\newblock On indecomposable trees in outer space.
\newblock {\em Geom. Dedicata}, 153:59-71, 2011.

\bibitem[Riv08]{rivin}
Igor Rivin.
\newblock Walks on groups, counting reducible matrices, polynomials, and
  surface and free group automorphisms.
\newblock {\em Duke Math. J.}, 142(2):353--379, 2008.


\bibitem[Sis11]{sisto}
Alessandro Sisto.
\newblock Contracting elements and random walks.
\newblock \arxivlink{1112.2666}, 2011.


\bibitem[Sta83]{stall-graphs}
John~R. Stallings.
\newblock Topology of finite graphs.
\newblock {\em Invent. Math.}, 71(3):551--565, 1983.

\bibitem[Vog02]{vogt-survey}
Karen Vogtmann.
\newblock Automorphisms of free groups and outer space.
\newblock {\em Geom. Dedicata}, 94:1--31, 2002.

\end{thebibliography}

\bigskip

\noindent\textsc{Thierry Coulbois}\\
LATP\\
\textsc{Avenue de l'escadrille Normandie-Niémen}\\
13013 \textsc{Marseille}\\
\textsc{France}\\
thierry.coulbois\string@univ-amu.fr

\bigskip

\noindent\textsc{Arnaud Hilion}\\
LATP\\
\textsc{Avenue de l'escadrille Normandie-Niémen}\\
13013 \textsc{Marseille}\\
\textsc{France}\\
arnaud.hilion\string@univ-amu.fr

\end{document}